\documentclass[1p]{elsarticle}

\usepackage{lineno}
\modulolinenumbers[5]

\journal{~~~}

\usepackage{mathrsfs}
\usepackage{amsfonts}
\usepackage{enumerate}
\usepackage{latexsym}
\usepackage[all,cmtip]{xy}

\biboptions{numbers,sort&compress}

\usepackage[colorlinks,linkcolor=blue,citecolor=blue]{hyperref}

\usepackage{amsmath,amssymb,xspace,amsthm}
\usepackage{mathtools}
\usepackage{bbm}

\numberwithin{equation}{section}

\def\fa{\mathfrak{a}}
\def\fg{\mathfrak{g}}

\def\fs{\mathfrak{s}}

\def\cK{\mathcal{K}}

\def\cT{\mathcal{T}}

\def\bC{\mathbb{C}}
\def\bN{\mathbb{N}}
\def\bZ{\mathbb{Z}}

\def\ptl{\partial}

\newtheorem{theo}{{Theorem}}[section]
\newtheorem{lemm}[theo]{Lemma}

\newtheorem{coro}[theo]{Corollary}
\newtheorem{prop}[theo]{Proposition}

\allowdisplaybreaks

\begin{document}

\begin{frontmatter}



\title{Classification of simple Harish-Chandra modules over the $N=1$ Ramond algebra}


\author{Yan-an Cai, Dong Liu and Rencai L\"{u}}

\begin{abstract}  In this paper, we give a new approach to classify all Harish-Chandra modules for the $N=1$ Ramond algebra $\mathfrak s$ based on the so called $A$-cover theory developed in \cite{BF}.

\end{abstract}


\end{frontmatter}


\section{Introduction}

Superconformal algebras have a long history in mathematical physics. The simplest examples, after the Virasoro algebra itself
(corresponding to $N = 0$) are the $N = 1$ superconformal algebras: the Neveu-Schwarz algebra  and the
Ramond algebra. These infinite dimensional Lie superalgebras are also called the super-Virasoro algebras as they can be regarded as natural  super generalizations of the Virasoro algebra.  Weight modules for the super-Viraoro algebras  have been extensively investigated (cf. \cite{DLM,IK1,IK2}), for more related results we refer the reader to \cite{Rao, Ka, Ka1, KL, KS, LPX1, LPX2, Ma, MZe, MZ, S1} and references therein.
 It is an important and challenging problem to give
complete classifications of Harish-Chandra modules (simple weight modules with finite dimensional weight spaces) for superconformal algebras.  In \cite{CP}, all simple unitary weight modules with finite dimensional weight spaces over the $N=1$ superconformal algebra were classified, which includes highest and lowest weight modules. Recently simple weight modules with finite dimensional weight spaces over the $N=2$ superconformal algebra were classified in \cite{LPX}.
 With the  theory of the $A$-cover in \cite{BF} for the Virasoro algebra, \cite{XL} completed such classification for the Lie superalgebra $W_{m, n}$ (also  see \cite{BFIK}). A complete classification for the $N=1$ superconformal algebra  was  given by Su in \cite{S0}. However, the complicated computations in the proofs make it extremely difficult to follow.
In this paper, we give a new approach to classify all Harish-Chandra modules for the $N=1$ Ramond algebra $\mathfrak s$ based on the $A$-cover theory.

This paper is arranged as follows. In Section 2, we recall some
notations and collect known facts about the $N=1$ Ramond algebra $\mathfrak s$.
In Section 3, we classify all simple cuspidal modules for $\mathfrak s$. With this classification we get the main result about the calssfication of Harish-Chandra modules over $\mathfrak s$ in Section 4.

Throughout this paper, we denote by $\bZ, \bZ_+, \bN, \bC$ and $\bC^*$ the sets of all integers, non-negative integers, positive integers, complex numbers, and nonzero complex numbers, respectively. All vector spaces and algebras in this paper are over $\bC$. We denote by $U(\fa)$ the universal enveloping algebra of the Lie superalgebra $\fa$ over $\bC$. Also, we denote by $\delta_{i,j}$ the Kronecker delta.

\section{Preliminaries}

In this section, we collect some basic definitions and results for our study.

A vector superspace $V$ is a vector space endowed with a $\bZ_2$-gradation $V=V_{\bar{0}}\oplus V_{\bar{1}}$. The parity of a homogeneous element $v\in V_{\bar{i}}$ is denoted by $|v|=\bar{i}\in\bZ_2$. Throughout this paper, when we write $|v|$ for an element $v\in V$, we will always assume that $v$ is a homogeneous element.

The $N=1$ Ramond algebra $\fs$ is the Lie superalgebra with basis $\{L_n,G_n,C\,|\, n\in\bZ\}$ and brackets
\begin{align*}
[L_m,L_n]&=(n-m)L_{m+n}+\delta_{m+n,0}\frac{1}{12}(n^3-n)C,\\
[L_m,G_p]&=(p-\frac{m}{2})G_{p+m},\\
[G_p,G_q]&=-2L_{p+q}+\delta_{p+q,0}\frac{4p^2-1}{12}C.
\end{align*}
The even part of $\fs$ is spanned by $\{L_n, C\,|\,n\in\bZ\}$, and is isomorphic to the Virasoro algebra, the universal central extension of the Witt algebra $\frak w$. The odd part of $\fs$ is spanned by $\{G_n\,|\, n\in\bZ\}$. Let $\bar{\fs}$ be the quotient algebra $\fs/\bC C$.

Let $A=\bC[t^{\pm1}]\otimes\Lambda(1)$, where $\Lambda(1)$ is the Grassmann algebra in one variable $\xi$. $A$ is $\bZ_2$-graded with $|t|=\bar{0}$ and $|\xi|=\bar{1}$. $A$ is an $\bar{\fs}$-module with
\begin{align*}
L_n\circ x&=t^{n+1}\ptl_t(x)+\frac{n}{2}t^n\xi\ptl_\xi(x),\\
G_n\circ x&=t^{n+1}\xi\ptl_t(x)-t^n\ptl_\xi(x),
\end{align*}
where $n\in\bZ,x\in A,\ptl_t=\frac{\ptl}{\ptl t},\ptl_\xi=\frac{\ptl}{\ptl \xi}$. So, we have Lie superalgebra $\widetilde{\fs}=\bar{\fs}\ltimes A$ with $A$ an abelian Lie superalgebra and $[x,y]=x\circ y, x\in\fs, y\in A$.

On the other hand, $\bar{\fs}$ has a natural $A$-module structure
\begin{equation}\label{A-action}
t^iL_n:=L_{n+i}, t^iG_n:=G_{n+i}, \xi L_n=\frac{1}{2}G_n, \xi G_n=0, \forall n,i\in\bZ.
\end{equation}
And $\bar{\fs}$ is an $\widetilde{\fs}$-module with adjoint $\bar{\fs}$-actions and $A$ acting as  \eqref{A-action}:
\begin{align*}
&[L_m, t^iL_n]-t^i[L_m,L_n]=[L_m,L_{n+i}]-(n-m)t^iL_{m+n}=iL_{m+n+i}=it^{m+i}L_n,\\
&[L_m,\xi L_n]-\xi[L_m,L_n]=[L_m,\frac{1}{2}G_n]-(n-m)\xi L_{m+n}=\frac{1}{4}mG_{m+n}=\frac{m}{2}t^m\xi L_n,\\
&[L_m,t^i G_n]-t^i[L_m,G_n]=[L_m,G_{n+i}]-(n-\frac{m}{2})t^iG_{m+n}=iG_{m+n+i}=it^{m+i}G_n,\\
&[L_m,\xi G_n]-\xi[L_m,G_n]=0,\\
&[G_m,t^iL_n]-t^i[G_m,L_n]=[G_m,L_{n+i}]+(m-\frac{n}{2})t^iG_{m+n}=\frac{i}{2}G_{m+n+i}=it^{m+i}\xi L_n,\\
&[G_m,t^iG_n]-t^i[G_m,G_n]=[G_m,G_{n+i}]+2t^iL_{m+n}=0,\\
&[G_m,\xi L_n]+\xi[G_m,L_n]=\frac{1}{2}[G_m,G_n]=-L_{m+n}=-t^mL_n,\\
&[G_m,\xi G_n]+\xi[G_m,G_n]=-2\xi L_{m+n}=-G_{m+n}=-t^mG_n.
\end{align*}

An $A\bar{\fs}$-module is an $\widetilde{\fs}$-module with $A$ acting associatively. Let $U=U(\widetilde{\fs})$ and $I$ be the left ideal of $U$ generated by $t^i\cdot t^j-t^{i+j}, t^0-1, t^i\cdot \xi-t^i\xi$ and $\xi\cdot \xi$ for all $i,j\in\bZ$. Then it is clear $I$ is an ideal of $U$. Let $\overline{U}=U/I$. Then the category of $A\bar{\fs}$-modules is naturally equivalent to the category of $\overline{U}$-modules.

Let $\fg$ be any of $\widetilde{\fs},\fs,\bar{\fs}$. A $\fg$-module $M$ is called a \emph{weight} module if the action of $L_0$ on $M$ is diagonalizable. Let $M$ be a weight $\fg$-module. Then $M=\bigoplus\limits_{\lambda\in\bC}M_\lambda$, where $M_\lambda=\{v\in M\,|\, L_0v=\lambda v\}$, called the weight space of weight $\lambda$. The support of $M$ is $\mathrm{Supp}(M):=\{\lambda\in\bC\,|\,M_\lambda\neq0\}$. A weight $\fg$-module is called \emph{cuspidal} or \emph{uniformly bounded} if the dimension of weight spaces of $M$ is uniformly bounded, that is there is $N\in\bN$ such that $\dim M_\lambda<N$ for all $\lambda\in\mathrm{Supp}(M)$. Clearly, if $M$ is simple, then $\mathrm{Supp}(M)\subseteq\lambda+\bZ$ for some $\lambda\in\bC$.

Let $\sigma: L\to L'$ be any homomorphism of Lie superalgebras or associative superalgebras, and $M$ be any $L'$-module. Then $M$ become an $L$-module, denoted by $M^\sigma$, under $x\cdot v:=\sigma(x)v, \forall x\in L, v\in M$. Denote by $T$ the automorphism of $L$ defined by $T(x):=(-1)^{|x|}x, \forall x\in L$. For any $L$-module $M$, $\Pi(M)$ is the module defined by a parity-change of $M$.

A module $M$ over an associative superalgebra $B$ is called \emph{strictly simple} if it is a simple module over the associative algebra $B$ (forgetting the $\bZ_2$-gradation).

We need the following result on tensor modules over tensor superalgebras.
\begin{lemm}[{\cite[Lemma 2.1, 2.2]{XL}}]\label{tensor}
Let $B,B'$ be associative superalgebras, and $M,M'$ be $B,B'$ modules, respectively.
\begin{enumerate}
\item $M\otimes M'\cong\Pi(M)\otimes\Pi(M'^T)$ as $B\otimes B'$-modules.
\item If in addition that $B'$ has a countable basis and $M'$ is strictly simple, then
\begin{enumerate}
\item Any $B\otimes B'$-submodule of $M\otimes M'$ is of the form $N\otimes M'$ for some $B$-submodule $N$ of $M$;
\item Any simple quotient of the $B\otimes B'$-module $M\otimes M'$ is isomorphic to some $\overline{M}\otimes M'$ for some simple quotient $\overline{M}$ of $M$.
\item $M\otimes M'$ is a simple $B\otimes B'$-module if and only if $M$ is a simple $B$-module.
\item If $V$ is a simple $B\otimes B'$-module containing a strictly simple $B'=\bC\otimes B'$ module $M'$, then $V\cong M\otimes M'$ for some simple $B$-module $M$.
\end{enumerate}
\end{enumerate}
\end{lemm}

Let $\cK$ be the Weyl superalgebra $A[\ptl_t,\ptl_\xi]$. All simple weight $\cK$-modules are classified in \cite{XL}.
\begin{lemm}[{\cite[Lemma 3.5]{XL}}]\label{Kmod}
Any simple weight $\cK$-module is isomorphic to some $A(\lambda)$ for some $\lambda\in\bC$ up to a parity-change, here $A(\lambda)\cong \cK/I_\lambda$ with $I_\lambda$ the left ideal of $\cK$ generated by $t\ptl_t-\lambda,\ptl_\xi$.
\end{lemm}

Also,  the following results about $(t-1)\bar{\fs}\subset \bar\fs$ follow from (\ref{A-action}) directly.
\begin{lemm}\label{rel-subalg}
Let $k,\ell\in\bZ_+$. Then for all $i,j\in\bZ$,
\begin{align*}
&[(t-1)^kL_i,(t-1)^\ell L_j]=(\ell-k+j-i)(t-1)^{k+\ell}L_{i+j}+(\ell-k)(t-1)^{k+\ell-1}L_{i+j},\\
&[(t-1)^kL_i,(t-1)^\ell G_j]=(j-\frac{i}{2})(t-1)^{k+\ell}G_{i+j}+(\ell-\frac{k}{2})(t-1)^{k+\ell-1}G_{i+j+1},\\
&[(t-1)^kG_i,(t-1)^\ell G_j]=-2(t-1)^{k+\ell}L_{i+j}.
\end{align*}
\end{lemm}

From Lemma \ref{rel-subalg}, we get
\begin{lemm}\label{ideal}
For $k\in\bN$, let $\fa_k=(t-1)^k\bar{\fs}$. Then
\begin{enumerate}
\item $\fa_1$ is a Lie subsuperalgebra of $\bar{\fs}$;
\item $\fa_k$ is an ideal of $\fa_1$ and $\fa_1/\fa_2$ is a two dimensional Lie superalgebra with bosonic basis $X$ and femionic basis $Y$ and nontrivial brackets $[X,Y]=\frac{1}{2}Y$.
\item The ideal generated by $\{(t-1)^kL_m\,|\, m\in\bZ\}$ is $\fa_k$.
\end{enumerate}
\end{lemm}

\begin{lemm}
Let $L=\bC X+\bC Y$ be the Lie superalgebra with $|X|=\bar{0}, |Y|=\bar{1}$ and $[X,Y]=\frac{1}{2}Y,[Y,Y]=0$. Then any simple finite dimensional $L$-module is one dimensional with $X.v=b v, Y.v=0$ for some $b\in\bC$.
\end{lemm}

\begin{lemm}[{\cite[Theorem 2.1]{M}}, Engel's Theorem for Lie superalgebras]\label{Engel}
Let $V$ be a finite dimensional module for the Lie superalgebra $L=L_{\bar{0}}\oplus L_{\bar{1}}$ such that the elements of $L_{\bar{0}}$ and $L_{\bar{1}}$ respectively are nilpotent endomorphisms of $V$. Then there exists a nonzero element $v\in V$ such that $xv=0$ for all $x\in L$.
\end{lemm}

\section{Cuspidal modules}
For $m\in\bZ\setminus\{0\}$, let
\begin{align*}
X_m&:=t^{-m}\cdot L_m+\frac{m}{2}t^{-m}\xi\cdot G_m-L_0,\\
Y_m&:=t^{-m}\cdot G_m-2t^{-m}\xi\cdot L_m-G_0+2\xi\cdot L_0\in \overline{U}.
\end{align*}
And let $\cT$ be the subspace of $\overline{U}$ spanned by $\{X_m,Y_m\,|\, m\in\bZ\setminus\{0\}\}$. Then we have
\begin{prop}
\begin{enumerate}
\item $[\cT,G_0]=[\cT,A]=0$.
\item $\cT$ is a Lie subsuperalgebra of $\overline{U}$. Moreover, $\cT$ is isomorphic to the Lie superalgebra $(t-1)\bar{\fs}$.
\end{enumerate}
\end{prop}
\begin{proof}
The first statement follows from
\begin{align*}
[G_0,X_m]&=[G_0,t^{-m}]\cdot L_m+t^{-m}\cdot[G_0,L_m]+\frac{m}{2}([G_0,t^{-m}\xi]\cdot G_m-t^{-m}\xi\cdot[G_0,G_m])\\
&=-mt^{-m}\xi\cdot L_m+\frac{m}{2}t^{-m}\cdot G_m+\frac{m}{2}(-t^{-m}\cdot G_m+2t^{-m}\xi\cdot L_m)\\
&=0,\\
\\
[G_0,Y_m]&=[G_0,t^{-m}]\cdot G_m+t^{-m}[G_0,G_m]-2([G_0,t^{-m}\xi]\cdot L_m-t^{-m}\xi\cdot[G_0,L_m])-[G_0,G_0]+2[G_0,\xi]\cdot L_0\\
&=-mt^{-m}\xi\cdot G_m-2t^{-m}\cdot L_m-2(-t^{-m}\cdot L_m-\frac{m}{2}t^{-m}\xi\cdot G_m)+2L_0-2L_0\\
&=0,\\
\\
[t^n,X_m]&=t^{-m}[t^n,L_m]+\frac{m}{2}t^{-m}\xi[t^n,G_m]+[L_0,t^n]=-nt^n+nt^n=0,\\
\\
[t^n,Y_m]&=t^{-m}[t^n,G_m]-2t^{-m}\xi[t^n,L_m]-[t^n,G_0]+2\xi[t^n,L_0]=-nt^n\xi+2nt^n\xi+nt^n\xi-2nt^n\xi=0,\\
\\
[X_m,\xi]&=t^{-m}[L_m,\xi]+\frac{m}{2}t^{-m}\xi[G_m,\xi]-[L_0,\xi]=\frac{m}{2}t^{-m}\xi-\frac{m}{2}t^{-m}\xi=0,\\
\\
[Y_m,\xi]&=t^{-m}[G_m,\xi]-2t^{-m}\xi[L_m,\xi]-[G_0,\xi]+2\xi[L_0,\xi]=-1+1=0.
\end{align*}
And the second statement follows from
\begin{align*}
[X_m,X_n]=&[t^{-m}\cdot L_m+\frac{m}{2}t^{-m}\xi\cdot G_m-L_0,t^{-n}\cdot L_n+\frac{n}{2}t^{-n}\xi\cdot G_n-L_0]\\
=&t^{-m}[L_m,t^{-n}]\cdot L_n-t^{-n}[L_n,t^{-m}]\cdot L_m+t^{-m-n}\cdot[L_m,L_n]\\
&+\frac{n}{2}\big(t^{-m}[L_m,t^{-n}\xi]\cdot G_n-t^{-n}\xi[G_n,t^{-m}]\cdot L_m+t^{-m-n}\xi\cdot[L_m,G_n]\big)\\
&-t^{-m}\cdot[L_m,L_0]+[L_0,t^{-m}]\cdot L_m\\
&+\frac{m}{2}\big(t^{-m}\xi[G_m,t^{-n}]\cdot L_n-t^{-n}[L_n,t^{-m}\xi]\cdot G_m+t^{-m-n}\xi\cdot[G_m,L_n]\big)\\
&+\frac{mn}{4}\big(t^{-m}\xi\cdot[G_m,t^{-n}\xi]\cdot G_n-t^{-n}\xi[G_n,t^{-m}\xi]\cdot G_m\big)-[L_0,t^{-n}]\cdot L_n\\
&-t^{-n}\cdot[L_0,L_n]-\frac{n}{2}\big([L_0,t^{-n}\xi]\cdot G_n+t^{-n}\xi\cdot[L_0,G_n]\big)\\
=&-nt^{-n}\cdot L_n+mt^{-m}\cdot L_m+(n-m)t^{-m-n}\cdot L_{m+n}\\
&+\frac{n}{2}\big((\frac{m}{2}-n)t^{-n}\xi\cdot G_n+(n-\frac{m}{2})t^{-m-n}\xi\cdot G_{m+n}\big)\\
&+\frac{m}{2}\big((m-\frac{n}{2})t^{-m}\xi\cdot G_m-(m-\frac{n}{2})t^{-m-n}\xi\cdot G_{m+n}\big)\\
&+\frac{mn}{4}(-t^{-n}\xi\cdot G_n+t^{-m}\xi\cdot G_m)\\
=&-nX_n+mX_m+(n-m)X_{m+n},\\
[X_m,Y_n]=&[t^{-m}\cdot L_m+\frac{m}{2}t^{-m}\xi\cdot G_m-L_0,t^{-n}\cdot G_n-2t^{-n}\xi\cdot L_n-G_0+2\xi\cdot L_0]\\
=&t^{-m}[L_m,t^{-n}]\cdot G_n-t^{-n}[G_n,t^{-m}]\cdot L_m+t^{-m-n}\cdot[L_m,G_n]\\
&-2\big(t^{-m}[L_m,t^{-n}\xi]\cdot L_n-t^{-n}\xi[L_n,t^{-m}]\cdot L_m+t^{-m-n}\xi\cdot[L_m,L_n]\big)-[t^{-m},G_0]\cdot L_m\\
&-t^{-m}\cdot[L_m.G_0]+2\big(t^{-m}[L_m,\xi]\cdot L_0-\xi[L_0,t^{-m}]\cdot L_m+t^{-n}\xi\cdot[L_m,L_0]\big)\\
&+\frac{m}{2}\big(t^{-m}\xi[G_m,t^{-n}]\cdot G_n-t^{-n}[G_n,t^{-m}\xi]\cdot G_m+t^{-m-n}\xi\cdot[G_m,G_n]\big)\\
&-m\big(t^{-m}\xi[G_m,t^{-n}\xi]\cdot L_n-t^{-n}\xi[L_n,t^{-m}\xi]\cdot G_m\big)-\frac{m}{2}\big(t^{-m}\xi\cdot[G_m,G_0]-[G_0,t^{-m}\xi]\cdot G_m\big)\\
&+m\big(t^{-m}\xi[G_m,\xi]\cdot L_0-\xi[L_0,t^{-m}\xi]\cdot G_m\big)-[L_0,t^{-n}]\cdot G_n\\
&-t^{-n}\cdot[L_0,G_n]+2[L_0,t^{-n}\xi]\cdot L_n+2t^{-n}\xi\cdot[L_0,L_n]\\
=&-nt^{-n}\cdot G_n+mt^{-m}\xi\cdot L_m+(n-\frac{m}{2})t^{-m-n}\cdot G_{m+n}\\
&-2\big((\frac{m}{2}-n)t^{-n}\xi\cdot L_n+mt^{-m}\xi\cdot L_m+(n-m)t^{-m-n}\xi\cdot L_{m+n}\big)-mt^{-m}\xi\cdot L_m\\
&+\frac{m}{2}t^{-m}\cdot G_m+2(\frac{m}{2}\xi\cdot L_0+mt^{-m}\xi\cdot L_m-mt^{-m}\xi\cdot L_m)\\
&+\frac{m}{2}(t^{-m}\cdot G_m-2t^{-m-n}\xi\cdot L_{m+n})+mt^{-n}\xi\cdot L_n-\frac{m}{2}(-2t^{-m}\xi\cdot L_m+t^{-m}\cdot G_m)\\
&-m\xi\cdot L_0+nt^{-n}\cdot G_n-nt^{-n}\cdot G_n-2nt^{-n}\xi\cdot L_n+2nt^{-n}\xi\cdot L_n\\
=&-nY_n+\frac{m}{2}Y_m+(n-\frac{m}{2})Y_{m+n},\\
&[t^{-m}\cdot G_m-2t^{-m}\xi\cdot L_m,t^{-n}\cdot G_n-2t^{-n}\xi\cdot L_n]\\
=&t^{-m}[G_m,t^{-n}]\cdot G_n+t^{-n}[G_n,t^{-m}]\cdot G_m+t^{-m-n}\cdot[G_m,G_n]\\
&-2\big(t^{-m}[G_m,t^{-n}\xi]\cdot L_n+t^{-n}\xi[L_n,t^{-m}]\cdot G_m-t^{-m-n}\xi\cdot[G_m,L_n]\big)\\
&-2\big(t^{-m}\xi[L_m,t^{-n}]\cdot G_n+t^{-n}[G_n,t^{-m}\xi]\cdot L_m+t^{-m-n}\xi\cdot[L_m,G_n]\big)\\
&+4\big(t^{-m}\xi[L_m,t^{-n}\xi]\cdot L_n+t^{-n}\xi[L_n,t^{-n}\xi]\cdot L_m\big)\\
=&-nt^{-n}\xi\cdot G_n-mt^{-m}\xi\cdot G_m-2t^{-m-n}\cdot L_{m+n}\\
&-2\big(-t^{-n}\cdot L_n-mt^{-m}\xi\cdot G_m+(m-\frac{n}{2})t^{-m-n}\xi\cdot G_{m+n}\big)\\
&-2\big(-nt^{-n}\xi\cdot G_n-t^{-m}\xi\cdot L_m+(n-\frac{m}{2}t^{-m-n}\xi\cdot G_{m+n})\big)\\
=&2(X_n+X_m-X_{m+n}+L_0),\\
[Y_m,Y_n]=&2(X_n+X_m-X_{n+m}).
\end{align*}

Moreover, $\cT$ is isomorphic to $(t-1)\bar{\fs}$ via $\varphi: \cT\mapsto(t-1)\bar{\fs}$; $X_m\mapsto L_m-L_0, Y_m\mapsto G_m-G_0$.
\end{proof}


\begin{prop}
We have the associative superalgebra isomorphism $\overline{U}\cong \cK\otimes U(\cT)$.
\end{prop}
\begin{proof}
Note that $U(\cT)$ is an associative subalgebra of $\overline{U}$ and the map $\tau: A[G_0]\to\cK$ with $\tau|_A=\mathrm{Id}_A, \tau(G_0)=\xi t\ptl_t-\ptl_\xi$ is a homomorphism of associative superalgebras. Define the map $\iota: A[G_0]\otimes U(\cT)\to\overline{U}$ by $\iota(t^i\xi^jG_0^k\otimes y)=t^i\xi^j\cdot G_0^k\cdot y+I, \forall i\in\bZ, j=0,1, k\in\bZ_+, y\in U(\cT)$. Then the restrictions of $\iota$ on $A[G_0]$ and $U(\cT)$ are well-defined homomorphisms of associative superalgebras. Also, note that $[\cT,A]=[\cT,G_0]=0$, $\iota$ is a well defined homomorphism of associative superalgebras. From
\begin{align*}
&\iota(t^m\otimes X_m-\frac{m}{2}t^m\xi\otimes Y_m+t^mL_0\otimes1-\frac{m}{2}t^m\xi G_0\otimes1)=L_m,\\
&\iota(t^m\otimes Y_m+2t^m\xi\otimes X_m+t^mG_0\otimes 1)=G_m,
\end{align*}
we can see that $\iota$ is surjective.

By PBW theorem we know that $\overline{U}$ has a basis consisting monomials in variables $\{L_m,G_m\,|\, m\in\bZ\setminus\{0\}\}$ over $A[G_0]$. Therefore $\overline{U}$ has an $A[G_0]$-basis consisting monomials in the variables $\{t^{-m}\cdot L_m-L_0, t^{-m}\cdot G_m-G_0\,|\,m\in\bZ\setminus\{0\}\}$. So $\iota$ is injective and hence an isomorphism.
\end{proof}

For any $(t-1)\bar{\fs}$-module $V$, we have the $A\bar{\fs}$-module $\Gamma(\lambda,V)=(A(\lambda)\otimes V)^{\varphi_1}$, where $\varphi_1: \overline{U}\xrightarrow{\iota^{-1}}\cK\otimes U(\cT)\xrightarrow{1\otimes\varphi}\cK\otimes U((t-1)\bar{\fs})$. More precisely, $\Gamma(\lambda,V)=A\otimes V$ with actions
\begin{align*}
t^i\xi^r.(y\otimes u):=&t^i\xi^ry\otimes u,\\
L_m.(y\otimes u):=&t^my\otimes(L_m-L_0).u-(-1)^{|y|}\frac{m}{2}t^m\xi y\otimes(G_m-G_0).u\\
&+t^m(\lambda y+t\ptl_t(y))\otimes u+\frac{m}{2}t^m\xi\ptl_\xi(y)\otimes u,\\
G_m.(y\otimes u):=&(-1)^{|y|}t^my\otimes(G_m-G_0).u+2t^m\xi y\otimes(L_m-L_0).u\\
&+t^m\xi (\lambda y+t\ptl_t(y))\otimes u-t^m\ptl_\xi(y)\otimes u.
\end{align*}

\begin{lemm}
\begin{enumerate}
\item For any $\lambda\in\bC$ and any simple $(t-1)\bar{\fs}$-module $V$, $\Gamma(\lambda,V)$ is a simple weight $A\bar{\fs}$-module.
\item Any simple weight $A\bar{\fs}$-module $M$ is isomorphic to some $\Gamma(\lambda,V)$ for some $\lambda\in\mathrm{Supp}(M)$ and some simple $(t-1)\bar{\fs}$-module $V$.
\end{enumerate}
\end{lemm}
\begin{proof}
The first statement follows from Lemma \ref{tensor} and Lemma \ref{Kmod}. For the second statement, let $M$ be any simple weight $A\bar{\fs}$-module with $\lambda\in\mathrm{Supp}(M)$. Then $M^{\varphi_1^{-1}}$ is a simple $\cK\otimes U((t-1)\bar{\fs})$-module. Fix a nonzero homogeneous element $v\in(M^{\varphi_1^{-1}})_\lambda$, then $\bC[\ptl_\xi]v$ is a finite dimensional supersubspace with $\ptl_\xi$ acting nilpotently. So there exists a nonzero element $v'$ in $\bC[\ptl_\xi]v$ with $I_\lambda v'=0$. Clearly,$\cK v'$ is isomorphic to $A(\lambda)$ or $\Pi(A(\lambda))$. Hence by Lemma \ref{tensor} and Lemma \ref{Kmod}, there exists a simple $U((t-1)\bar{\fs})$-module $N$ such that $M^{\varphi_1^{-1}}\cong A(\lambda)\otimes N$ or $M^{\varphi_1^{-1}}\cong\Pi(A(\lambda))\otimes N\cong A(\lambda)\otimes\Pi(N^T)$.
\end{proof}

Thus, to classify all simple weight $A\bar{\fs}$-modules, it suffices to classify all simple $(t-1)\bar{\fs}$-modules. In particular, to classify all simple cuspidal $A\bar{\fs}$-modules, it suffices to classify all finite dimensional simple $(t-1)\bar{\fs}$-modules.

\begin{lemm}
\begin{enumerate}
\item Let $V$ be any finite dimensional $(t-1)\bar{\fs}$-module. Then there exists $k\in\bN$ such that $\fa_kV=0$.
\item Let $V$ be any simple finite dimensional simple $(t-1)\bar{\fs}$-module. Then $\fa_2V=0$. In particular, $\dim V=1$.
\end{enumerate}
\end{lemm}
\begin{proof}
\begin{enumerate}
\item Since $V$ is a finite dimensional $(t-1)\frak w$-module, there exists $k\in\bN$ such that $(t-1)^k\frak wV=0$. So the first statement follows from Lemma \ref{ideal}.
\item Consider the finite dimensional Lie superalgebra $\fg=\fa_1/\mathrm{ann}V$, then $V$ is a finite dimensional $\fg_{\bar{0}}$-module and $\fa_{2,\bar{0}}+\mathrm{ann}(V)$ acts nilpotently on $V$. Since $[x,x]\in\fa_{2,\bar{0}}$ for all $x\in\fa_{2,\bar{1}}$, every element in $\fa_{2,\bar{1}}+\mathrm{ann}(V)$ acts nilpotently on $V$. Hence, by Lemma \ref{Engel}, there is nonzero $v\in V$ annihilated by $\fa_2+\mathrm{ann}(V)$. And therefore $\fa_2V=0$, which means $V$ is a simple finite dimensional module for $\fa_1/\fa_2$.\qedhere
\end{enumerate}
\end{proof}

\begin{coro}\label{inter}
Any simple cuspidal $A\bar{\fs}$-module is isomorphic to some $\Gamma(\lambda,b)=A\otimes \bC u$ with $\lambda,b\in\bC$ defined as follows:
\begin{align*}
t^i\xi^r.(y\otimes u)&=t^i\xi^ry\otimes u,\\
L_m.(t^i\xi^r\otimes u)&=(\lambda+i+m(b+\frac{1}{2}\delta_{\bar{1},\bar{r}}))t^{m+i}\xi^r\otimes u,\\
G_m.(t^i\otimes u)&=(\lambda+i+2mb)t^{m+i}\xi\otimes u,\\
G_m.(t^i\xi\otimes u)&=-t^{m+i}\otimes u,
\end{align*}
where $i,m\in\bZ, r=0,1, y\in A$.
\end{coro}

Next we are going to define the $A$-cover $\widehat{M}$ of a cuspidal $\bar{\fs}$-module $M$. Consider $\bar{\fs}$ as the adjoint $\bar{\fs}$-module. Then the tensor product $\bar{\fs}$-module $\bar{\fs}\otimes M$ is an $A\bar{\fs}$-module by
\[
x\cdot(y\otimes b):=(xy)\otimes v, \forall x\in A, y\in\bar{\fs}, v\in M.
\]

Let $K(M)=\{\sum\limits_i x_i\otimes v_i\in\bar{\fs}\otimes M\,|\,\sum\limits_i (ax_i)v_i=0, \forall a\in A\}$. Then $K(M)$ is an $A\bar{\fs}$-submodule of $\bar{\fs}\otimes M$. And hence we have the $A\bar{\fs}$-module $\widehat{M}=(\bar{\fs}\otimes M)/K(M)$, called the \emph{cover} of $M$ when $\bar{\fs}M=M$, as in \cite{BF}.  Clearly, the linear map $\pi: \widehat{M}\to\bar{\fs}M; x\otimes v+K(M)\mapsto xv$ is an $\bar{\fs}$-module epimorphism.

Recall that in \cite{BF}, the authors show that every cuspidal $W$-module is annihilated by the operators $\Omega_{k,s}^{(m)}$ for $m$ large enough.

\begin{lemm}[{\cite[Corollary 3.7]{BF}}]\label{Omegaoper}
For every $\ell\in\bN$ there exists $m\in\bN$ such that for all $k, s\in\bZ$ the differentiators $\Omega_{k, s}^{(m)}=\sum\limits_{i=0}^m(-1)^i\binom{m}{i}L_{k-i}L_{s+i}$ annihilate every cuspidal $W$-module with a composition series of length $\ell$.
\end{lemm}

Let $M$ be a cuspidal $\bar{\fs}$-module. Then $M$ is a cuspidal $W$-module and hence there exists $m\in\bN$ such that $\Omega_{k,p}^{(m)}M=0, \forall k,p\in\bZ$. Therefore, $[\Omega_{k,p}^{(m)},G_j]M=0, \forall j,k,p\in\bZ, s\in S$. Thus, on $M$ we have
\begin{align*}
0=&[\Omega_{k,p-1}^{(m)},G_{j+1}]-2[\Omega_{k,p}^{(m)},G_j]+[\Omega_{k,p+1}^{(m)},G_{j-1}]-[\Omega_{k+1,p-1}^{(m)},G_{j}]\\
&+2[\Omega_{k+1,p}^{(m)},G_{j-1}]-[\Omega_{k+1,p+1}^{(m)},G_{j-2}]\\
=&[\sum\limits_{i=0}^m(-1)^i\binom{m}{i}L_{k-i}L_{p-1+i},G_{j+1}]-2[\sum\limits_{i=0}^m(-1)^i\binom{m}{i}L_{k-i}L_{p+i},G_j]\\
&+[\sum\limits_{i=0}^m(-1)^i\binom{m}{i}L_{k-i}L_{p+1+i},G_{j-1}]-[\sum\limits_{i=0}^m(-1)^i\binom{m}{i}L_{k+1-i}L_{p-1+i},G_{j}]\\
&+2[\sum\limits_{i=0}^m(-1)^i\binom{m}{i}L_{k+1-i}L_{p+i},G_{j-1}]-[\sum\limits_{i=0}^m(-1)^i\binom{m}{i}L_{k+1-i}L_{p+1+i},G_{j-2}]\\
=&\sum\limits_{i=0}^{m}(-1)^i\binom{m}{i}\Big((j+1-\frac{k-i}{2})G_{k-i+j+1}L_{p-1+i}+(j+1-\frac{p-1+i}{2})L_{k-i}G_{p+i+j}\\
&-2(j-\frac{k-i}{2})G_{k-i+j}L_{p+i}-2(j-\frac{p+i}{2})L_{k-i}G_{p+i+j}+(j-1-\frac{k-i}{2})G_{k-i+j-1}L_{p+i+1}\\
&+(j-1-\frac{p+i+1}{2})L_{k-i}G_{p+i+j}-(j-\frac{k-i+1}{2})G_{k-i+j+1}L_{p+i-1}\\
&-(j-\frac{p+i-1}{2})L_{k-i+1}G_{p+i+j-1}+2(j-1-\frac{k-i+1}{2})G_{k-i+j}L_{p+i}\\
&+2(j-1-\frac{p+i}{2})L_{k-i+1}G_{p+i+j-1}-(j-2-\frac{k-i+1}{2})G_{k-i+j-1}L_{p+i+1}\\
&-(j-2-\frac{p+i+1}{2})L_{k+1-i}G_{p+i+j-1}\Big)\\
=&\frac{3}{2}\sum\limits_{i=0}^{m}(-1)^i\binom{m}{i}(G_{k-i+j+1}L_{p+i-1}-2G_{k-i+j}L_{p+i}+G_{k-i+j-1}L_{p+i+1})\\
=&\frac{3}{2}\sum\limits_{i=0}^m(-1)^i\binom{m+2}{i}G_{k-i+j+1}L_{p+i-1}.
\end{align*}
That is, we have
\begin{lemm}\label{omega}
Let $M$ be a cuspidal $\bar{\fs}$-module. Then there exists $m\in\bN$ such that for all $j,p\in\bZ$ the operators $\overline{\Omega}_{j,p}^{(m)}=\sum\limits_{i=0}^m(-1)^i\binom{m}{i}G_{j-i}L_{p+i}$ annihilate $M$.
\end{lemm}

\begin{lemm}\label{covercuspidal}
For any cuspidal $\bar{\fs}$-module $M$, $\widehat{M}$ is also cuspidal.
\end{lemm}
\begin{proof}
Since $\widehat{M}$ is an $A$-module, it suffices to show that one of its weight spaces is finite dimensional. Fix a weight $\alpha+p, p\in\bZ$ and let us prove that $\widehat{M}_{\alpha+p}=\mathrm{span}\{L_{p-k}\otimes M_{\alpha+k}, G_{p-k}\otimes M_{\alpha+k}\,|\,k\in\bZ\}$ is finite dimensional. Assume that $\alpha=0$ when $\alpha+\bZ=\bZ$. From Lemma \ref{Omegaoper} and Lemma \ref{omega}, there exists $m\in\bN$, such that $\sum\limits_{i=0}^m(-1)^i\binom{m}{i}L_{j-i}L_{p+i}v=\sum\limits_{i=0}^m(-1)^i\binom{m}{i}G_{j-i}L_{p+i}v=0, \forall j,p\in\bZ, v\in M$. Hence,
\begin{equation}\label{induction}
\sum\limits_{i=0}^m(-1)^i\binom{m}{i}L_{j-i}\otimes L_{p+i}v,\sum\limits_{i=0}^m(-1)^i\binom{m}{i}G_{j-i}\otimes L_{p+i}v\in K(M).
\end{equation}

We are going to prove by induction on $|q|$ for $q\in\bZ$ that for all $u\in M_{\alpha+q}$,
\[
L_{p-q}\otimes u,G_{p-q}\otimes u\in\sum\limits_{|k|\leq\frac{m}{2}}\Big(L_{p-k}\otimes M_{\alpha+k}+G_{p-k}\otimes M_{\alpha+k}\Big)+K(M).
\]
We only need to prove this claim for $|q|>\frac{m}{2}$, and we may assume that $q<-\frac{m}{2}$, the proof for $q>\frac{m}{2}$ is similar. Since $L_0$ acts on $M_{\alpha+q}$ with a nonzero scalar, we can write $u=L_0v$ for some $v\in M_{\alpha+q}$. Then by \eqref{induction} and induction hypothesis, we have
\begin{align*}
L_{p-q}\otimes L_0v&=\sum\limits_{i=0}^m(-1)^i\binom{m}{i}L_{p-q-i}\otimes L_iv-\sum\limits_{i=1}^m(-1)^i\binom{m}{i}L_{p-q-i}\otimes L_iv\\
&\in\sum\limits_{|k|\leq\frac{m}{2}}\Big(L_{p-k}\otimes M_{\alpha+k}+G_{p-k}\otimes M_{\alpha+k}\Big)+K(M),\\
G_{p-q}\otimes L_0v&=\sum\limits_{i=0}^m(-1)^i\binom{m}{i}G_{p-q-i}\otimes L_iv-\sum\limits_{i=1}^m(-1)^i\binom{m}{i}G_{p-q-i}\otimes L_iv\\
&\in\sum\limits_{|k|\leq\frac{m}{2}}\Big(L_{p-k}\otimes M_{\alpha+k}+G_{p-k}\otimes M_{\alpha+k}\Big)+K(M).\qedhere
\end{align*}
\end{proof}

Now we can classify all simple cuspidal $\bar{\fs}$-modules.
\begin{theo}\label{cuspidal}
Any nontrivial simple cuspidal $\bar{\fs}$-module is isomorphic to a simple quotient of $\Gamma(\lambda, b)$ for some $\lambda,b\in\bC$.
\end{theo}
\begin{proof}
Let $M$ be any nontrivial simple cuspidal $\bar{\fs}$-module. Then $\bar{\fs}M=M$ and there is an epimorphism $\pi: \widehat{M}\to M$. From Lemma \ref{covercuspidal}, $\widehat{M}$ is cuspidal. Hence $\widehat{M}$ has a composition series of $A\bar{\fs}$-submodules:
\[
0=\widehat{M}^{(0)}\subset\widehat{M}^{(1)}\subset\cdots\subset\widehat{M}^{(s)}=\widehat{M}
\]
with $\widehat{M}^{(i)}/\widehat{M}^{(i-1)}$ being simple $A\bar{\fs}$-modules. Let $k$ be the minimal integer such that $\pi(\widehat{M}^{(k)})\neq0$. Then we have $\pi(\widehat{M}^{(k)})=M, \widehat{M}^{(k-1)}=0$ since $M$ is simple. So we have an $\bar{\fs}$-epimorphism from the simple $A\bar{\fs}$-module $\widehat{M}^{(k)}/\widehat{M}^{(k-1)}$ to $M$. Now theorem follows from Corollary \ref{inter}.
\end{proof}

\section{Main results}

In this section, we will classify all simple weight $\fs$-modules with finite dimensional weight spaces. First of all, from the representation theory of Virasoro algebra, we know that $C$ acts trivially on any simple cuspidal $\fs$-module, and hence the category of simple cuspidal $\fs$-modules is naturally equivalent to the category of simple cuspidal $\bar{\fs}$-modules. Thus, it remains to classify all all simple weight $\fs$-modules with finite dimensional weight spaces which is not cuspidal. From now on, we will assume $M$ is such an $\fs$-module. Let $\lambda\in\mathrm{supp}(M)$.

The following result is well-known
\begin{lemm}\label{weightupper}
Let $M$ be a weight module with finite dimensional weight spaces for the Virasoro algebra with $\mathrm{supp}(M)\subseteq\lambda+\bZ$. If for any $v\in M$, there exists $N(v)\in\bN$ such that $L_iv=0, \forall i\geq N(v)$, then $\mathrm{supp}(M)$ is upper bounded.
\end{lemm}

\begin{lemm}\label{appN=1}
Suppose $M$ is a simple weight $\fs$-module with finite dimensional weight spaces which is not cuspidal, then $M$ is a highest (or lowest) weight module.
\end{lemm}
\begin{proof}
Since $M$ is not cuspidal, then there is a $k\in\bZ$ such that $\dim M_{-k+\lambda}>2(\dim M_\lambda+\dim M_{\lambda+1})$. Without lost of generality, we may assume that $k\in\bN$. Then there exists a nonzero element $w\in M_{-k+\lambda}$ such that $L_kw=L_{k+1}w=G_kw=G_{k+1}w=0$. Therefore, $L_iw=G_iw=0$ for all $i\geq k^2$, since $[\fs_i,\fs_j]=\fs_{i+j}$.

It is easy to see that $M'=\{v\in M\,|\,\dim\fs^+v<\infty\}$ is a nonzero submodule of $M$, here $\fs^+=\sum\limits_{n\in\bN}(\bC L_n+\bC G_n)$. Hence $M=M'$. So, Lemma \ref{weightupper} tells us that $\mathrm{supp}(M)$ is upper bounded, that is $M$ is a highest weight module.
\end{proof}

Combining with Lemma \ref{appN=1} and Theorem \ref{cuspidal}, we can get the following result, which was given in \cite{S0} by much complicated calculations.

\begin{theo}\label{ts2}
Let $V$ be a simple ${\frak s}$-module with finite dimensional weight spaces. Then $V$ is a highest weight module, a lowest weight module, or a simple quotient of $\Gamma(\lambda, b)$ for some $\lambda,b\in\bC$ (which is called a module of the intermediate series).
\end{theo}

\noindent{\bf Acknowledgement:} Y. Cai is partially supported by NSF of China (Grant 11801390). D. Liu is partially supported by NSF of China (Grant 11971315£¬11871249). R. L\"{u} is partially supported by NSF of China (Grant 11471233, 11771122, 11971440).


\

Y. Cai: Department of Mathematics, Soochow University, Suzhou 215006, P. R. China. Email: yatsai@mail.ustc.edu.cn

\

D. Liu: Department of Mathematics, Huzhou University, Zhejiang Huzhou, 313000, China. Email:  liudong@zjhu.edu.cn

\

R. L\"{u}: Department of Mathematics, Soochow University, Suzhou 215006, P. R. China. Email: rlu@suda.edu.cn

\end{document}